\documentclass[a4paper, 11pt]{amsart} 
\usepackage[headings]{fullpage}               
\usepackage{boxedminipage}
\usepackage{amsfonts}
\usepackage{amsmath} 
\usepackage{amssymb}
\usepackage{graphicx}
\usepackage{amsthm}
\usepackage{subfig}
\usepackage{tikz}
\usepackage{setspace}
\usepackage{hyperref}
\usepackage{enumerate}
\hypersetup{linktocpage=true,}
\newtheorem{theorem}{Theorem}[section]
\newtheorem{lemma}[theorem]{Lemma}
\newtheorem{proposition}[theorem]{Proposition}
\newtheorem{corollary}[theorem]{Corollary}
\theoremstyle{definition}\newtheorem{definition}[theorem]{Definition}
\theoremstyle{definition}\newtheorem{example}[theorem]{Example}
\theoremstyle{definition}\newtheorem{remark}[theorem]{Remark}

\newcommand{\rank}{\operatorname{rank}}

\newcommand{\smooth}{\operatorname{smooth}(X)}
\newcommand{\Iso}{\operatorname{Isom}}

\newcommand{\Isolin}{\operatorname{Isom^{Lin}}}

\newcommand{\Full}{\operatorname{Full}}
\newcommand{\stab}{\operatorname{Stab}}

\begin{document}

\title{Equivalence of continuous, local and infinitesimal rigidity in normed spaces}

\author{Sean Dewar
}

\begin{abstract}
We present a rigorous study of framework rigidity in general finite dimensional normed spaces from the perspective of Lie group actions on smooth manifolds. As an application, we prove an extension of Asimow and Roth's 1978/9 result establishing the equivalence of local, continuous and infinitesimal rigidity for regular bar-and-joint frameworks in a $d$-dimensional Euclidean space. Further, we obtain upper bounds for the dimension of the space of trivial motions for a framework and establish the flexibility of small frameworks in general non-Euclidean normed spaces.

\keywords{Bar-joint frameworks \and Infinitesimal rigidity \and Continuous rigidity \and Local rigidity \and Finite dimensional normed spaces}
\end{abstract}

\maketitle
\tableofcontents

\section{Introduction}

A framework $(G,p)$ is an embedding $p$ of the vertices of a simple graph $G$ into a given normed space. With a given framework a natural question is whether it is in some sense ``rigid''. In Euclidean spaces many types of rigidity such as \textit{global}, \textit{redundant} and \textit{universal} have been studied intensely \cite{jordanbook} \cite{globalshinichi} \cite{universal}. We wish to detect whether a framework in a general normed space is structurally rigid in the sense that either any continuous motion of the vertices that preserves the edge lengths corresponds to an isometric motion of the embedded vertices (\textit{continuous rigidity}) or that the embedding is locally unique up to an isometric map (\textit{local rigidity}). In Euclidean space, one method is to consider the a priori stronger notion of \textit{infinitesimal rigidity} as this implies local and continuous rigidity (see for example \cite{gluck}). We shall prove the below theorem (found in Section \ref{3}):

\begin{theorem}\label{asimowroth}
Let $(G,p)$ be a constant (see Section \ref{3}) finite framework in a finite dimensional normed space $X$, then the following are equivalent:
\begin{enumerate}[(i)]
\item $(G,p)$ is infinitesimally rigid in $X$,

\item $(G,p)$ is locally rigid in $X$,

\item $(G,p)$ is continuously rigid in $X$.
\end{enumerate}
\end{theorem}

In 1978/9 L.~Asimow and B.~Roth proved Theorem \ref{asimowroth} in the special case where $X$ is Euclidean \cite{asiroth} \cite{asirothtwo}. Recently research has been undertaken into framework rigidity in non-Euclidean normed spaces, in particular spaces with $\ell_p$ norms ($p \in [1,\infty]$) \cite{noneuclidean}, polyhedral norms \cite{polyhedra} and matrix norms such as the Schatten $p$-norms \cite{matrixnorm}. For recent research into infinitesimal rigidity we refer the reader to \cite{symmetry} \cite{jordansymmetry} for frameworks with symmetry, \cite{molecular} \cite{bodybar} for infinitesimal rigidity concerning alternative types of frameworks and 
\cite{twotwo} for frameworks on surfaces.

In Section \ref{trivial section} we will present a rigorous study of the orbit and the trivial motion space of a set of points. We will give an upper bound for the dimension of the space of trivial motions which will be achievable by most placements. Utilising this we shall in Section \ref{3} prove Theorem \ref{asimowroth}. 

In Section \ref{small section} we shall obtain further bounds on the dimension of the space of trivial infinitesimal motions for frameworks that lie on some hyperplane of the normed space. These results will allow us to prove that no \textit{small} framework (a framework with less vertices than the dimension of the normed space plus one) on two or more vertices is infinitesimally rigid in a non-Euclidean space (see Theorem \ref{smallflexible}).

\section{Preliminaries}\label{preliminaries}

All normed spaces $(X, \| \cdot \|)$ shall be assumed to be over $\mathbb{R}$ and finite dimensional; further we shall denote a normed space by $X$ when there is no ambiguity. For any normed space $X$ we shall use the notation $B_r(x)$, $B_r[x]$ and $S_r[x]$ for the open ball, closed ball and the sphere with centre $x$ and radius $r > 0$ respectively. We shall define a normed space to be a \textit{Euclidean space} if its norm is generated by an inner product, otherwise $X$ is a \textit{non-Euclidean (normed) space}. 

Given normed spaces $X,Y$ we shall denote by $L(X,Y)$ the normed space of all linear maps from $X$ to $Y$ with the operator norm $\|\cdot\|_{\text{op}}$ and $A(X, Y)$ to be space of all affine maps from $X$ to $Y$ with the norm topology. If $X=Y$ we shall abbreviate to $L(X)$ and $A(X)$ and if $Y= \mathbb{R}$ with the standard norm we define $X^*:= L(X, \mathbb{R})$ and refer to the operator norm as $\| \cdot \|$ when there is no ambiguity. We denote by $\iota$ the identity map on $X$.

For a $C^1$-differentiable manifold $M$ we shall denote by $T_x M$ the tangent space of $M$ at $x \in M$ and $TM := \bigcup_{x \in M} T_x M$. For a general reference on the theory of manifolds we refer the reader to \cite[Section 3]{manifold}.

\subsection{Differentiation}

For normed spaces $X$, $Y$ and $U \subset X$, $V \subset Y$ we define a map $f:U \rightarrow V$ to be \textit{(Fr\'{e}chet) differentiable} at $x_0 \in U^\circ$ (the interior of $U$) if there exists a linear map $df(x_0):X \rightarrow Y$ such that
\begin{align*}
\frac{\|f(x_0+h) - f(x_0) - df(x_0)h\|_Y}{\|h\|_X} \rightarrow 0
\end{align*}
as $h \rightarrow 0$; we refer to $df(x_0)$ as the \textit{(Fr\'{e}chet) derivative of $f$ at $x_0$}. If $U' \subset U^\circ$ is open and $f$ is differentiable at all points in $U'$ we say that $f$ is \textit{differentiable on $U'$}. If $f$ is differentiable on $U'$ and the map
\begin{align*}
df : U' \rightarrow L(X,Y), ~ x \mapsto df(x)
\end{align*}
is continuous then we say that $f$ is \text{$C^1$-differentiable} on $U'$ and define $df$ to be the \textit{$C^1$-derivative} of $f$; if $U'=U$ we just say that $f$ is $C^1$-differentiable. For all $k \in \mathbb{N}$ we define inductively $d^k f := d (d^{k-1} f)$ where $d^1f :=df$ and $d^0f := f$; by this we define $f$ to be \textit{$C^k$-differentiable} if $d^k f$ exists and is continuous. If $f$ is $C^k$-differentiable for all $k \in \mathbb{N} \cup \{0\}$ we say $f$ is \textit{$C^\infty$-differentiable} or \textit{smooth}. If $f$ is $C^k$-differentiable and bijective with $C^k$-differentiable inverse we say that $f$ is a \textit{$C^k$-diffeomorphism} or \textit{smooth diffeomorphism} if $k = \infty$. 

Some of the results referenced refer specifically to \textit{G\^{a}teaux differentiation}, however we will only consider Lipschitz maps between finite dimensional normed spaces and in this case G\^{a}teaux differentiability is equivalent to differentiability by \cite[Proposition 4.3]{nonlinear}. If we were to observe the connection between the rigidity map and the rigidity operator of an infinite framework this would not hold to be true.

For $C^k$-manifolds $M$ and $N$ with $k \in \mathbb{N} \cup \{0, \infty \}$ we define a map $f :M \rightarrow N$ to be \text{$C^k$-differentiable} if for any $x \in M$ and chart $(V, \psi)$ of $N$ with $f(x) \in N$ there exists a chart $(U, \phi)$ of $M$ such that $x \in U$, $f(U) \subset V$ and $\psi \circ f \circ \phi^{-1}$ is $C^k$-differentiable on $U$. If $f$ also has a $C^k$-differentiable inverse then it is a \textit{$C^k$-diffeomorphism}. For $k \geq 1$ we may define for each $x \in M$ the maps $df(x) : T_x M \rightarrow T_{f(x)} N$ and $df : TM \rightarrow TN$ that will be consistent if $M,N$ are normed spaces, we refer the reader to \cite[Section 3.3]{manifold} for more detail. 

\begin{remark}
If we have a continuous path $\alpha :(a,b) \rightarrow X$ that is differentiable at $t \in (a,b)$ with differential $\alpha'(t)$ in the traditional sense i.e
\begin{align*}
\alpha'(t) := \lim_{h \rightarrow 0}  \frac{\alpha(t+h) - \alpha(t)}{h},
\end{align*}
then $\alpha'(t) = d \alpha (t) (1)$.
\end{remark}

\subsection{Support functionals, smoothness and strict convexity}

Let $x \in X$ and $f \in X^*$, then we say that $f$ is \textit{support functional} of $x$ if $\|f\| = \|x\|$ and $f(x) = \|x\|^2$. By an application of the Hahn-Banach theorem it can be shown that every point must have a support functional. 

We say that a non-zero point $x$ is \textit{smooth} if it has a unique support functional and define $\smooth \subseteq X \setminus \{0\}$ to be the set of smooth points of $X$. If $\smooth \cup \{0\} =X$ then we say that $X$ is \textit{smooth}. We define a norm to be \textit{strictly convex} if $\|tx +(1-t) y\| < 1$ for all distinct $x,y \in S_1[0]$ and $t \in (0,1)$.

The \textit{dual map} of $X$ is the map $\varphi : \smooth \cup \{0\} \rightarrow X^*$ that sends each smooth point to its unique support functional and $\varphi(0)=0$. It is immediate that $\varphi$ is homogeneous since $f$ is the support functional of $x$ if and only if $a f$ is the support functional of $ax$ for $a \neq 0$.

\begin{remark}
If $X$ is Euclidean with inner product $\left\langle \cdot , \cdot \right\rangle$ then all non-zero points are smooth and we have $\varphi(x) = \left\langle x, \cdot \right\rangle$ where $\left\langle x, \cdot \right\rangle : y \mapsto \left\langle x, y \right\rangle$.
\end{remark}

\begin{proposition}\label{paper1}
For any normed space $X$ the following properties hold:
\begin{enumerate}[(i)]
\item \label{paper1item1} For $x_0 \neq 0$, $x_0 \in \smooth$ if and only if $x \mapsto \|x\|$ is differentiable at $x_0$.

\item \label{paper1item3} If $x \mapsto \|x\|$ is differentiable at $x_0$ then it has derivative $\frac{1}{\|x_0\|}\varphi(x_0)$.

\item \label{paper1item0} The set $\smooth$ is dense in $X$ and $\smooth^c$ has measure zero with respect to the Lebesgue measure on $X$

\item \label{paper1item2} The map $\varphi$ is continuous.
%
%
\end{enumerate}
\end{proposition}

\begin{proof} (\ref{paper1item1}) \& (\ref{paper1item3}): By \cite[Lemma 1]{maxwell}, $x \mapsto \|x\|$ is differentiable at $x_0$ if and only if $x_0 \in \smooth$ with derivative $\frac{1}{\|x_0\|}\varphi(x_0)$.

(\ref{paper1item0}): The result follows from \ref{paper1item1} and \cite[Theorem 25.5]{rockafellar} as $x \mapsto \|x\|$ is convex.

(\ref{paper1item2}): By \cite[Theorem 25.5]{rockafellar}, the map $x \mapsto \frac{1}{\|x\|}\varphi(x)$ is continuous on $\smooth$, thus $\varphi$ is continuous on $\smooth$ also. As $\varphi(x) = \| x\|$ it follows that $\varphi$ is continuous at $0 \in X^*$ also as required.
\end{proof}

\subsection{Isometry groups}

We shall define $\Iso (X, \| \cdot \|)$ to be the \textit{group of isometries} of $(X, \| \cdot \|)$ and $\Isolin (X, \| \cdot \|)$ to be the \textit{group of linear isometries} of $X$ with the group actions being composition; we shall denote these as $\Iso (X)$ and $\Isolin (X)$ if there is no ambiguity. It can be seen by Mazur-Ulam's theorem \cite{minkowski} that all isometries of a finite dimensional normed space are affine i.e.~each isometry is the unique composition of a linear isometry followed by a translation, thus $\Iso (X)$ has the topology inherited from $A(X)$. It follows from the Closed Subgroup theorem \cite[Theorem 5.1.14]{manifold} that for any normed space the group of isometries is a \textit{Lie group} (a smooth finite dimensional manifold with smooth group operations) while the group of linear isometries is a compact Lie group since it is closed and bounded in $L(X)$.

\begin{lemma}\label{paperisolem}
Let $X$ be a $d$-dimensional normed space, then the following holds:
\begin{enumerate}[(i)]
\item \label{paperisolemitem1} There exists a unique Euclidean space $(X, \| \cdot \|_2)$ such that $\Iso (X, \| \cdot\|)$ is a subgroup and a closed smooth submanifold of $\Iso(X, \| \cdot\|_2)$ and $\Isolin (X, \| \cdot\|)$ is a subgroup and a closed smooth submanifold of $\Isolin(X, \| \cdot\|_2)$.

\item \label{paperisolemitem2} If $X$ is Euclidean then:
\begin{enumerate}[(a)]
\item $\dim \Iso (X) = \frac{d(d+1)}{2}$,

\item $\dim \Isolin (X) = \frac{d(d-1)}{2}$.
\end{enumerate}

\item \label{paperisolemitem3} If $X$ is non-Euclidean then:
\begin{enumerate}[(a)]
\item $d \leq \dim \Iso (X) \leq \frac{d(d-1)}{2} +1$,

\item $0 \leq \dim \Isolin (X) \leq \frac{(d-1)(d-2)}{2} +1$.
\end{enumerate}
\end{enumerate}
\end{lemma}

\begin{proof}
(\ref{paperisolemitem1}): This follows from \cite[Corollary 3.3.4]{minkowski} and the Closed Subgroup theorem \cite[Theorem 5.1.14]{manifold}.

(\ref{paperisolemitem2}): See \cite[Section 2.5.5]{lie} and \cite[Section 2.5.9]{lie}.

(\ref{paperisolemitem3}): \cite[Lemma 4]{spheres}.

\end{proof}

\subsection{Placements and bar-joint frameworks}
We shall assume that all graphs are simple i.e.~no loops or parallel edges, however we will allow them to have a countably infinite vertex set unless we explicitly state otherwise. We will denote $V(G)$ and $E(G)$ to be the vertex and edge sets of $G$ respectively. If $H$ is a subgraph of $G$ we will represent this by $H \subseteq G$. For a set $S$ we shall denote by $K_S$ the complete graph on the set $S$.

Let $X$ be a normed space. For any set $S$ we say $p \in X^S$ is a \textit{placement of $S$ in $X$}; we will denote this $(p,S)$ if we need to clarify what set $p$ is the placement of. For a graph $G$ we say $p$ is a \textit{placement of $G$ in $X$} if $p$ is a placement of $V(G)$. We define a \textit{(bar-joint) framework} to be a pair $(G,p)$ where $G$ is a graph and $p$ is a placement of $G$ in $X$. For all $X$ and $S$ we will gift $X^S$ the product topology from $X$; if $|S|< \infty$ we define the norm
\begin{align*}
\|\cdot \|_S : (x_v)_{v \in S} \mapsto \max_{v \in S} \|x_v\| 
\end{align*}
on $X^S$. For $x \in X^S$ and $T \subset S$ we define $x|_{T} := (x_v)_{v \in T} \in X^T$.

A placement $p$ is \textit{spanning} in $X$ if the set $\{p_v: v\in S\}$ affinely spans $X$. A placement $p$ is in \textit{general position} if for any choice of distinct vertices $v_0, v_1, \ldots, v_n \in S$ ($n \leq \dim X$) the set $\{p_{v_i} : i = 0,1 ,\ldots,n\}$ is affinely independent. It is immediate that if $p$ is in general position and $|S| \geq \dim X +1$ then $p$ is spanning. We denote the set of placements of $S$ in general position by $\mathcal{G}(S) \subseteq X^S$; likewise for any graph $G$ we let $\mathcal{G}(G) := \mathcal{G}(V(G))$. If $S$ is finite then $\mathcal{G}(S)$ is an open dense subset of $X^S$ and $X^S \setminus \mathcal{G}(S)$ has measure zero; we can see this as $\mathcal{G}(S)$ is the complement of an algebraic set. 

For placements $(q,T)$, $(p,S)$ we say $(q,T)$ is a \textit{subplacement} of $(p,S)$ (or $(q,T) \subseteq (p,S)$) if $T \subseteq S$ and $p_v = q_v$ for all $v \in T$. For frameworks $(H,q)$ and $(G,p)$ we say $(H,q)$ is a \textit{subframework} of $(G,p)$ (or $(H,q) \subseteq (G,p)$) if $H \subseteq G$ and $p_v = q_v$ for all $v \in V(H)$. If $H$ is also a spanning subgraph we say that $(H,q)$ is a \textit{spanning subframework} of $(G,p)$.

\section{Trivial motions of placements}\label{trivial section}

\subsection{Structure of the orbit of a placement}

Let $\Gamma$ be a Lie group and $M$ a (finite dimensional) smooth manifold. If there exists a smooth group action 
\begin{align*}
\phi: \Gamma \times M \rightarrow M, ~ (g,x) \mapsto g.x
\end{align*}
we say that $\phi$ is a \textit{Lie group action of $\Gamma$ on $M$}. We define the following for all $x \in M$:
\begin{enumerate}[(i)]
\item the \textit{stabiliser} of $x$, $\stab_x := \{g \in \Gamma : g.x = x\}$,

\item the \textit{orbit} of $x$, $\mathcal{O}_x := \{g.x : g \in \Gamma \}$,

\item $\phi_x: \Gamma \rightarrow \mathcal{O}_x, ~ g \mapsto g.x$.
\end{enumerate}
We say that $\Gamma$ acts \textit{properly} on $M$ if the map 
\begin{align*}
\theta: \Gamma \times M \rightarrow M \times M, ~ (g, x) \mapsto (\phi(g,x),x)
\end{align*}
is proper i.e.~the preimage of any compact set is compact. If $H$ is a closed subgroup of $\Gamma$ then by \cite[Theorem 5.1.16]{manifold} $\Gamma/H$ (the set of left cosets $gH$, $g \in \Gamma$) has a unique manifold structure such that the quotient map $\pi : \Gamma \rightarrow \Gamma/H$ is a smooth surjective \textit{submersion} i.e.~$d\pi(g)$ is surjective for all $g \in \Gamma$.

\begin{lemma}\cite[Corollary 4.1.22]{mechanics}\label{orbits}
Let $\phi$ be a Lie group action of $\Gamma$ on $M$. Suppose $\Gamma$ acts properly on $M$, then $\mathcal{O}_x$ is a closed smooth submanifold of $M$ that is diffeomorphic to $\Gamma/\stab_x$ under the map $\tilde{\phi}_x : g\stab_x \mapsto g.x$.
\end{lemma}

For any set $S$, $x \in X^S$ and affine map $g \in A(X)$ we define $g.x := (g(x_v))_{v \in S}$. With this notation we define for any $S$ the map
\begin{align*}
\phi: \Iso (X) \times X^S \rightarrow X^S , ~ x \mapsto g.x .
\end{align*}
If $|S|< \infty$ then this is a Lie group action of $\Iso (X)$ on $X^S$; we shall always refer to this group action if we mention $\Iso (X)$ acting on $X^S$.

\begin{lemma}\label{proper}
For any $X$ and $|S| < \infty$ the group of isometries $\Iso (X)$ acts properly on $X^S$.
\end{lemma}

\begin{proof}
Let $((g_n.p^n,p^n))_{n \in \mathbb{N}}$ be a convergent sequence in the image of $\theta : \Iso(X) \times X^S \rightarrow X^S \times X^S$ with limit $(q,p)$. By Mazur-Ulam's theorem \cite{minkowski} that for each $n \in \mathbb{N}$ there exists $G_n \in \Isolin(X)$ and $x_n \in X$ such that $g_n = T_{x_n} \circ G_n$, where $T_{x_n}$ is the translation map $y \mapsto y+x_n$. As $((g_n.p^n,p^n))_{n \in \mathbb{N}}$ converges then $(g_n.p^n)_{n \in \mathbb{N}}$ and $(p^n)_{n \in \mathbb{N}}$ are bounded in $X^S$, thus $(x_n)_{n \in \mathbb{N}}$ is bounded as
\begin{align*}
\|x_n\| = \|(x_n)_{v \in S} \|_S \leq \|G_n.p^n + (x_n)_{v \in S} \|_S + \|G_n.p^n\|_S = \|g_n.p^n\|_S +\|p^n\|_S;
\end{align*}
it follows by Bolzano-Weierstrass that we have a convergent subsequence $(x_{n_k})_{k \in \mathbb{N}}$ with limit $x \in X$. Since $\Isolin(X)$ is compact, there exists a convergent subsequence $(G_{n_{k_l}})_{l \in \mathbb{N}}$ of $(G_{n_k})_{k \in \mathbb{N}}$ with limit $G \in \Isolin(X)$; this implies $(g_{n_{k_l}})_{l \in \mathbb{N}}$ converges to $g := T_x \circ G$. As $((g_n,p^n))_{n \in \mathbb{N}}$ has a convergent subsequence it follows that $\theta$ is proper as required.
\end{proof}

\begin{lemma}\label{homeo2.0}
Let $p$ be a placement of a finite set in $X$, then $\mathcal{O}_p$ is a closed smooth submanifold of $X^{V(G)}$ and the map
\begin{align*}
\tilde{\phi}_p : \Iso (X)/\stab_p \rightarrow \mathcal{O}_p, ~ g \stab_p \mapsto g.p
\end{align*}
is a smooth diffeomorphism.
\end{lemma}

\begin{proof}
By Lemma \ref{proper} and Lemma \ref{orbits} it follows that $\mathcal{O}_p$ is a closed smooth submanifold of $X^{V(G)}$ diffeomorphic to $\Iso (X)/\stab_p$ under the diffeomorphism $\tilde{\phi}_p$.
\end{proof}

\begin{lemma}\label{diffeo}
Let $(q,T) \subset (p,S)$ be placements in $X$ where the affine span of $\{p_v :v \in S\}$ is equal to the affine span of $\{q_v :v \in T\}$. If $|T| < \infty$ then $\mathcal{O}_p$ is a smooth manifold that is diffeomorphic to $\mathcal{O}_q$ and the restriction map 
\begin{align*}
\rho:\mathcal{O}_p \rightarrow \mathcal{O}_q, ~ (x_v)_{v \in S} \mapsto (x_v)_{v \in T}
\end{align*}
is a smooth diffeomorphism. 
\end{lemma}

\begin{proof}
Define the (finite dimensional) linear spaces
\begin{align*}
A_p := \{ h.p : h \in A(X) \} \qquad A_q := \{ h.q : h \in A(X) \}.
\end{align*}
We note that the linear map $\tilde{\rho}: A_p \rightarrow A_q$ where $\tilde{\rho}(x) := x|_{T}$ is a continuous linear isomorphism. This implies the map $\tilde{\rho}^{-1}|_{\mathcal{O}_q}$ is a smooth embedding into $A_p$ with image $\mathcal{O}_p$. By Lemma \ref{homeo2.0}, $\mathcal{O}_q$ is a smooth manifold, thus $\mathcal{O}_q$ is diffeomorphic to $\mathcal{O}_p$ and $\rho:= \tilde{\rho}|_{\mathcal{O}_p}^{\mathcal{O}_q}$ is a smooth diffeomorphism. 
\end{proof}

We will define a continuous path through a placement $p$ in $X^S$ to be a family $\alpha := (\alpha_v)_{v \in S}$ of continuous paths $\alpha_v : (-\delta,\delta) \rightarrow X$ (for some fixed $\delta >0$) where $\alpha_v(0) = p_v$ for all $v \in S$. If $\alpha(t):= (\alpha_v(t))_{v \in S} \in \mathcal{O}_p$ for all $t \in (-\delta,\delta)$ then $\alpha$ is a \textit{trivial finite motion}. 

Let $u \in X^S$. If there exists a trivial finite motion $\alpha$ of $p$ that is differentiable at $t=0$ and $u_v = \alpha'_v (0)$ for all $v \in S$ then we say that $u$ is a \textit{trivial (infinitesimal) motion of $p$}. For any placement $p$ we shall denote $\mathcal{T}(p)$ to be the the set all trivial infinitesimal motions of $p$.

\begin{theorem}\label{homeo2}
Let $p$ be a placement in $X$, then $\mathcal{O}_p$ is a smooth manifold with tangent space $\mathcal{T}(p)$ at $p$ and 
\begin{align*}
\tilde{\phi}_p : \Iso (X)/\stab_p \rightarrow \mathcal{O}_p, ~ g \stab_p \mapsto g.p
\end{align*}
is a smooth diffeomorphism.
\end{theorem}

\begin{proof}
Choose a finite subplacement $(q,T)$ of $(p,S)$ so that the $p$ and $q$ affinely span the same space, then by Lemma \ref{homeo2.0}, $\mathcal{O}_q$ is a smooth manifold diffeomorphic to $\Iso (X)/\stab_q$ under the smooth diffeomorphism $\tilde{\phi}_q$. By Lemma \ref{diffeo}, $\mathcal{O}_p$ is a smooth manifold diffeomorphic to $\Iso (X)/\stab_p$ and the restriction map $\rho : \mathcal{O}_p \rightarrow \mathcal{O}_q$ is a smooth diffeomorphism. As $\tilde{\phi}_p = \rho^{-1} \circ \tilde{\phi}_q$ then it is also a smooth diffeomorphism. It follows from its definition that $\mathcal{T}(p)$ is the tangent space of $\mathcal{O}_p$ at $p$.
\end{proof}

\begin{corollary}\label{homeo2.1}
The map $\phi_p$ is a smooth submersion and $d \phi_p(\iota) : T_\iota \Iso(X) \rightarrow \mathcal{T}(p)$ is surjective with $d \phi_p(\iota) g= g.p$ for all $g \in T_\iota \Iso (X)$. Further, $\ker d \phi_p(\iota) = T_\iota \stab_p$.
\end{corollary}

\begin{proof}
By Theorem \ref{homeo2}, $\tilde{\phi}_p$ is a smooth diffeomorphism. We note that $\phi_p = \tilde{\phi}_p \circ \pi$ where $\pi : \Iso(X) \rightarrow \Iso(X)/\stab_p$ is the natural quotient map. By the Closed Subgroup theorem \cite[Theorem 5.1.14]{manifold} $\pi$ is a smooth submersion, thus $\phi_p$ is a smooth submersion also and $d \phi_p(\iota)$ is surjective. As $\tilde{\phi}_p$ is a smooth diffeomorphism then $\ker d \phi_p(\iota) = \ker \pi$ as required.
\end{proof}

\begin{corollary}\label{diffeo2}
Let $(p,S)$ and $(q,T)$ be placements in $X$ where the affine span of $\{p_v :v \in S\}$ is equal to the affine span of $\{q_v :v \in T\}$, then the following hold:
\begin{enumerate}[(i)]
\item \label{diffeo2item1} The orbits $\mathcal{O}_p$ and $\mathcal{O}_q$ are diffeomorphic.

\item \label{diffeo2item2} $\dim \mathcal{T}(p)=\dim \mathcal{T}(q)$.

\item \label{diffeo2item3} If $(q,T) \subseteq (p,S)$ then the restriction map 
\begin{align*}
\rho:\mathcal{O}_p \rightarrow \mathcal{O}_q, ~ (x_v)_{v \in S} \mapsto (x_v)_{v \in T}
\end{align*}
is a smooth diffeomorphism.
\end{enumerate}
\end{corollary}

\begin{proof}
(\ref{diffeo2item3}): Choose a finite subplacement $(r, U) \subseteq (q,T) \subseteq (p,S)$ so that the affine span of $\{r_v :v \in U\}$ is equal to the affine span of $\{q_v :v \in T\}$. Define the restriction maps $\rho_T : \mathcal{O}_q \rightarrow \mathcal{O}_r$ and $\rho_S : \mathcal{O}_p \rightarrow \mathcal{O}_r$, then by Lemma \ref{diffeo}, $\rho_T, \rho_S$ are smooth diffeomorphisms. As $\rho = \rho_T^{-1} \circ \rho_S$ then it is also a smooth diffeomorphism.

(\ref{diffeo2item1}): If $(q,T) \subseteq (p,S)$ this follows from \ref{diffeo2item3}. Suppose $(q,T)$ is not a subplacement of $(p,S)$. Define $(t,S \sqcup T) $ to be the placement where $t|_{S}=p$ and $t|_{T} = q$. We note all three placements have the same affine span of their placement points. As $(p,S) \subset (t, S \sqcup T)$ and $(q,T) \subseteq (t, S \sqcup T)$ then by part \ref{diffeo2item3} we have $\mathcal{O}_p \cong \mathcal{O}_{t} \cong \mathcal{O}_q$ as required.

(\ref{diffeo2item2}): By Theorem \ref{homeo2}, $\mathcal{O}_p$ has tangent space $\mathcal{T}(p)$ at $p$ and $\mathcal{O}_q$ has tangent space $\mathcal{T}(q)$ at $q$. By part \ref{diffeo2item1}, $\mathcal{O}_p \cong \mathcal{O}_q$ and so $\dim \mathcal{T}(p)=\dim \mathcal{T}(q)$.
\end{proof}


\subsection{Upper and lower bounds of the dimension of $\mathcal{T}(p)$}

\begin{proposition}\label{homeo2.2}
The following hold for any placement $p$:
\begin{enumerate}[(i)]
\item \label{homeo2.2item1} $d \phi_p(\iota)$ is injective if and only if $\phi_p$ is a smooth local diffeomorphism i.e.~$d \phi_p(g)$ is bijective for all $g \in \Iso(X)$.

\item \label{homeo2.2item2} $\phi_p$ is injective if and only if $\phi_p$ is a smooth diffeomorphism. 
\end{enumerate}
If either \ref{homeo2.2item1} or \ref{homeo2.2item2} hold then $\dim \mathcal{T}(p) = \Iso(X)$.
\end{proposition}

\begin{proof}
If $\phi_p$ is a local diffeomorphism it follows $d \phi_p(\iota)$ is bijective. Suppose $d \phi_p(\iota)$ is injective. Choose $g \in \Iso(X)$, then $\dim \ker d\phi_p(g) = 0$; this follows as there exists non-zero $u \in \ker d\phi_p(g)$ with corresponding smooth curve $\alpha:(-1,1) \rightarrow \Iso (X)$ (i.e.~$\alpha(0)=g$, $\alpha'(0) =u$) then the curve $t \mapsto g^{-1}\alpha(t)$ generates a non-zero tangent vector at $\iota$ that lies in the kernel of $d \phi_p(\iota)$. By this it follows that $d \phi_p(\iota)$ is injective if and only if $d\phi_p(g)$ is injective for all $g \in \Iso (X)$. By Corollary \ref{homeo2.1} it follows that $d \phi_p(g)$ is bijective for all $g \in \Iso (X)$ and so is a local diffeomorphism.

If $\phi_p$ is injective then $\stab_p$ is trivial, thus $\Iso(X)/\stab_p =\Iso(X)$ and $\tilde{\phi}_p = \phi_p$. By Theorem \ref{homeo2} it then follows $\phi_p$ is a smooth diffeomorphism. Conversely suppose $\phi_p$ is a smooth diffeomorphism, then as $\phi_p = \tilde{\phi}_p \circ \pi$ it follows that the quotient map $\pi$ is a diffeomorphism. This implies $\pi$ is a group isomorphism, thus $\stab_p$ is trivial and $\phi_p$ is injective.

If either \ref{homeo2.2item1} or \ref{homeo2.2item2} hold then $d \phi_p(\iota)$ is bijective and $\dim \mathcal{T}(p) = \Iso(X)$.
\end{proof}

\begin{definition}
We define a placement $p$ to be \textit{full} if $\phi_p$ is a local diffeomorphism and \textit{isometrically full} if $\phi_p$ is a diffeomorphism. 
\end{definition}

It is immediate that any isometrically full placement is full. By part \ref{homeo2.2item1} of Proposition \ref{homeo2.2} our notion of full agrees with that given in \cite{matrixnorm}. The set of full placements of a set $S$ will be denoted by $\Full (S)$ and likewise the set of full placements of a graph $G$ will be denoted by $\Full (G)$.

\begin{corollary}\label{spanfull}
All spanning placements are isometrically full.
\end{corollary}

\begin{proof}
Suppose $g.p=p$ and choose $v_0, \ldots, v_d \in S$ so that $p_{v_0}, \ldots, p_{v_d}$ is an affine basis of $X$, then $g(p_{v_i})=p_{v_i}$ for all $i=0, \ldots,d$. By Mazum-Ulam's theorem \cite{minkowski} $g$ is affine and so since $p_{v_0}, \ldots, p_{v_d}$ is an affine basis of $X$ this map must be unique. As $\iota.p=p$ then $g = \iota$ and $\phi_p$ is injective. The result now follows by part \ref{homeo2.2item2} of Proposition \ref{homeo2.2}.
\end{proof}

\begin{example}
We shall denote by $X$ the space $\mathbb{R}^2$ with the $\ell_3$-norm. The linear isometries of $X$ are generated by the $\pi/2$ anticlockwise rotation around the origin and the reflection in the line $\{(t,0):t \in \mathbb{R} \}$. Let $S = \{v_1,v_2\}$ and $(p,S)$ and $(q,S)$ be the non-spanning placements in $X$ where $p_{v_1}=q_{v_1}=0$, $p_{v_2}=(1,0)$ and $q_{v_2} = (1,2)$. Both placements are full in $X$, however $q$ is isometrically full while $p$ is not. This example shows that while all spanning placements are isometrically full and all isometrically full placements are full the reverse is not necessarily true.
\end{example}

\begin{proposition}\label{openfull}
Let $d +1 \leq |S|< \infty$ and $\dim X =d$. Then $\Full(S)$ is an open dense subset of $X^S$ and $\Full(S)^c$ has measure zero with respect to the Lebesgue measure on $X^S$.
\end{proposition}

\begin{proof}
Since $|S| \geq d+1$ then all placements in general position are spanning. By Corollary \ref{spanfull} we have $\mathcal{G}(S) \subset \Full(S)$, thus as $\mathcal{G}(S)$ is dense in $X^S$ then $\Full(S)$ is dense in $X^S$. Since $\mathcal{G}(S)^c$ is an algebraic set then it has measure zero, thus it follows $\Full(S)^c$ also has measure zero.

Define the affine map 
\begin{align*}
F: X^S \rightarrow L(T_{\iota} \Iso (X),X^S), ~ p \mapsto d\phi_p(\iota).
\end{align*}
The set of injective maps of $L(T_{\iota} \Iso (X),X^S)$ is open. We note $\Full(S)$ is the preimage of the set of injective maps of $L(T_{\iota} \Iso (X),X^S)$ under $F$ by part \ref{homeo2.2item2} of Proposition \ref{homeo2.2} and so $\Full(S)$ is open. 
\end{proof}

\begin{corollary}
All isometrically full placements in $X$ are spanning if and only if $X$ is Euclidean.
\end{corollary}

\begin{proof}
Suppose all isometrically full placements in $X$ are spanning, then it follows by part \ref{homeo2.2item2} of Proposition \ref{homeo2.2} that for all linear hyperplanes $Y$ of $X$ there exists a linear map $T_Y \neq \iota$ that is invariant on $Y$. By \cite[(4.7)]{euclidean} $X$ is Euclidean.

Conversely suppose $X$ is Euclidean. If $p$ is a non-spanning placement then $p$ lies in some affine hyperplane $H$. We note that if $h$ is the reflection in $H$ then $h.p =p$, thus by part \ref{homeo2.2item2} of Proposition \ref{homeo2.2} $p$ is not isometrically full.
\end{proof}

We may now give an upper and lower bound for the dimension of $\mathcal{T}(p)$.

\begin{theorem}\label{flextheorem}
For any placement $p$ in a $d$-dimensional space $X$,
\begin{align*}
d \leq \dim \mathcal{T}(p) \leq \dim \Iso(X)
\end{align*}
with $\dim \mathcal{T}(p) = \dim \Iso(X)$ if and only if $p$ is full.
\end{theorem}

\begin{proof}
By Corollary \ref{homeo2.1}, $d \phi_p(\iota) : T_\iota \Iso (X) \rightarrow \mathcal{T}(p)$ is surjective, thus we have $\dim \mathcal{T}(p) \leq \dim \Iso (X)$. Let $x_1, \ldots, x_d \in X$ be a basis and define for each $i \in \{1, \ldots, d\}$ the trivial finite flex $\alpha(i)$ where for each $v \in S$ we have $\alpha(i)_v : (-1,1) \rightarrow X$, $t \mapsto p_v + tx_i$. We note that $(\alpha(i)'_v(0))_{v \in S} = (x_i)_{v \in S} \in \mathcal{T}(p)$ for each $i \in \{1, \ldots, d\}$, thus $\dim \mathcal{T}(p) \geq d$.

If $p$ is full then by Proposition \ref{homeo2.2}, $\dim \mathcal{T}(p) = \dim \Iso(X)$. If $\dim \mathcal{T}(p) = \dim \Iso(X)$ then by Corollary \ref{homeo2.1}, $d \phi_p (\iota)$ is bijective; it then follows by part \ref{homeo2.2item1} of Proposition \ref{homeo2.2} that $p$ is full.
\end{proof}

A \textit{rigid motion} of $X$ is a family $\gamma :=(\gamma_x)_{x \in X}$ of continuous maps $\gamma_x : (-\delta,\delta) \rightarrow X , ~ x \in X$ (for some fixed $\delta >0$) where $\gamma_x(0) =x$ and $\|\gamma_x(t) - \gamma_y(t) \| = \|x - y\|$ for all $x, y \in X$ and $t \in (-\delta, \delta)$. The following shows that our definition of a trivial finite motion agrees with the definition given in \cite{noneuclidean} if a framework is isometrically full.

\begin{proposition}\label{isotriv}
Let $p$ be an isometrically full placement in $ X$. If $\alpha$ is a continuous path through $p$ in $X^S$ then the following are equivalent:
\begin{enumerate}[(i)]
\item \label{isotriv1} $\alpha$ is a trivial finite motion.

\item \label{isotriv2} There exists a unique continuous path $h:(-\delta, \delta) \rightarrow \Iso (X)$ such that $h_t(p_v)= \alpha_v(t)$ for all $t \in (-\delta,\delta)$, $v \in S$.

\item \label{isotriv3} There exists a unique rigid motion $\gamma$ such that $\gamma_{p_v} = \alpha_v$ for all $v \in S$.
\end{enumerate}
\end{proposition}

\begin{proof}
(\ref{isotriv1} $\Rightarrow$ \ref{isotriv2}): As $\alpha$ is a continuous path in $\mathcal{O}_p$ and $\phi_p$ is a smooth diffeomorphism we define the unique continuous path $h := \phi^{-1}_p \circ \alpha$.

(\ref{isotriv2} $\Rightarrow$ \ref{isotriv3}): Define $\gamma$ to be the unique family of maps $\gamma$ where $\gamma_x(t) = h_t(x)$ for all $x \in X$ and $t \in (-\delta, \delta)$, then $\gamma$ is a rigid motion as required. 

(\ref{isotriv3} $\Rightarrow$ \ref{isotriv1}): We note that $\gamma$ restricted to the set $\{p_v : v\in S\}$ is a trivial finite motion.
\end{proof}

\section{Equivalence of continuous, local and infinitesimal rigidity}\label{3}

\subsection{Continuous, local and infinitesimal rigidity}

We say that an edge $vw \in E(G)$ of a framework $(G,p)$ is \textit{well-positioned} if $p_v -p_w \in \smooth$; if this holds we define $\varphi_{v,w} := \varphi \left(\frac{p_v-p_w}{\|p_v-p_w\|} \right)$. If all edges of $(G,p)$ are well-positioned we say that $(G,p)$ is \textit{well-positioned} and $p$ is a \textit{well-positioned placement} of $G$. We shall denote the subset of well-positioned placements of $G$ in $X$ by the set $\mathcal{W}(G)$. 

\begin{lemma}\label{wellpos}
Let $G$ be finite, then $\mathcal{W}(G)$ is dense subset of $X^{V(G)}$ and $\mathcal{W}(G)^c$ has measure zero with respect to the Lebesgue measure on $X^{V(G)}$.
\end{lemma}

\begin{proof}
By part \ref{paper1item0} of Proposition \ref{paper1} the set $\smooth$ is dense and its compliment has measure zero, thus the result holds for all graphs with a single edge. Suppose the result holds for all graphs with $n-1$ edges and let $G$ be any graph with $n$ edges. Choose $vw \in E(G)$, and define $G_1,G_2$ to be the graphs on $V(G)$ where $E(G_1) = E(G) \setminus \{vw\}$ and $E(G_2) = \{vw\}$ then $\mathcal{W}(G_1)^c$ and $\mathcal{W}(G_2)^c$ have measure zero by assumption. As $\mathcal{W}(G)^c = \mathcal{W}(G_1)^c \cup \mathcal{W}(G_2)^c$ then $\mathcal{W}(G)^c$ has measure zero also; this further implies $\mathcal{W}(G)$ is also dense. The result now follows by induction.
\end{proof}

We define the \textit{rigidity map} of $G$ (in $X$) to be the continuous map
\begin{align*}
f_G: X^{V(G)} \rightarrow \mathbb{R}^{E(G)}, ~ x = (x_v)_{v \in V(G)} \mapsto (\|x_v - x_w \|)_{vw \in E(G)}
\end{align*}
and for well-positioned placements $p$ we also define the \textit{rigidity operator} of $G$ at $p$ in $X$ to be the continuous linear map
\begin{align*}
df_G(p): X^{V(G)} \rightarrow \mathbb{R}^{E(G)}, ~ x = (x_v)_{v \in V(G)} \mapsto (\varphi_{v,w}(x_v - x_w ))_{vw \in E(G)}.
\end{align*}
For any framework we define the \textit{configuration space} of $(G,p)$ in $X$ to be the set $f_G^{-1} [f_G(p)]$.

\begin{proposition}\cite[Proposition 6]{maxwell}\label{paper2}
If $G$ is finite then $f_G$ is differentiable at $p$ if and only if $p$ is a well-positioned placement of $G$; if this holds then the rigidity operator at $p$ is the derivative of the rigidity map at $p$.
\end{proposition}

\begin{lemma}\label{rigopcont}
If $G$ is finite then the map
\begin{align*}
df_G : \mathcal{W}(G) \rightarrow L(X^{V(G)}, \mathbb{R}^{E(G)}), ~ x \mapsto df_G(x)
\end{align*}
is continuous.
\end{lemma}

\begin{proof}
This follows from part \ref{paper1item2} of Proposition \ref{paper1}.
\end{proof}

For a finite graph $G$ we say that a well-positioned framework $(G,p)$ is \textit{regular} if for all $q \in \mathcal{W}(G)$ we have $\rank df_G(p) \geq \rank df_G(q)$. We shall denote the subset of $\mathcal{W}(G)$ of regular placements of $G$ by $\mathcal{R}(G)$.

\begin{lemma}\label{regopen2}
Let $G$ be finite, then $\mathcal{R}(G)$ is a non-empty open subset of $\mathcal{W}(G)$.
\end{lemma}

For this lemma we shall need to use the fact that the rank function on the space of linear maps between finite dimensional normed spaces $X$, $Y$ is \textit{lower semi-continuous} i.e.~for all $c \geq 0$ the set $\{ T \in L(X,Y) : \rank T \geq c \}$ is open.

\begin{proof}
Let $n := \sup \{\rank df_G(p) : p \in \mathcal{W}(G) \}$. The rank function $T \mapsto \rank T$ is lower semi-continuous and by Lemma \ref{rigopcont}, $df_G$ is continuous, thus the map 
\begin{align*}
f: \mathcal{W}(G) \rightarrow \mathbb{N}, ~ p \mapsto \rank df_G(p)
\end{align*}
is lower semi-continuous. As $\mathcal{R}(G) = f^{-1}[[n , \infty)]$ then $\mathcal{R}(G)$ is open.
\end{proof}

We define a \textit{finite flex} of a framework $(G,p)$ to be a continuous path $\alpha$ through a placement $p$ where $\| \alpha_v(t) - \alpha_w(t) \| = \|p_v - p_w \|$ for all $vw \in E(G)$ and $t \in (-\delta,\delta)$. If $\alpha$ is a trivial finite motion of a placement $p$ of $G$ we say $\alpha$ is a \textit{trivial finite flex} of $(G,p)$; we note that $\alpha$ will automatically be a finite flex of $(G,p)$ as isometries preserve distance.

If the only finite flexes of $(G,p)$ are trivial then $(G,p)$ is \textit{continuously rigid (in $X$)}; $(G,p)$ will be defined to be \textit{continuously flexible} if it is not continuously rigid. For a finite framework $(G,p)$ we say $(G,p)$ is \textit{locally rigid (in $X$)} if there exists a neighbourhood $U \subseteq X^{V(G)}$ of $p$ such that $f^{-1}_G[f^{-1}_G(p)] \cap U = \mathcal{O}_p \cap U$; likewise we shall define a framework to be \textit{locally flexible} if it is not locally rigid. We classify these as types of \textit{finite rigidity}.

We define $u \in X^{V(G)}$ to be a \textit{trivial (infinitesimal) flex of $(G,p)$} if $u$ is a trivial motion of $p$. If $(G,p)$ is well-positioned we say that $u \in X^{V(G)}$ is an \textit{(infinitesimal) flex of $(G,p)$} if $df_G(p) u =0$. The following proposition shows a link between finite and infinitesimal flexes for frameworks. 

\begin{lemma}\label{fininflink}
Let $(G,p)$ be a well-positioned framework in $X$ and $\alpha$ a finite flex of $(G,p)$ that is differentiable at $0$, then $(\alpha'_v(0))_{v \in V(G)}$ is an infinitesimal flex of $(G,p)$.
\end{lemma} 

\begin{proof}
This follows from the proof of \cite[Lemma 2.1.(ii)]{noneuclidean}.
\end{proof}

Since all trivial flexes of $(G,p)$ are trivial motions of $p$ we shall also denote $\mathcal{T}(p)$ to be the set all trivial infinitesimal flexes $(G,p)$. If $(G,p)$ is well-positioned we define $\mathcal{F}(G,p)$ to be the space of all infinitesimal flexes of $(G,p)$. The latter is clearly a linear space as it is exactly the kernel of the rigidity operator. By Proposition \ref{fininflink} it follows $\mathcal{T}(p) \subseteq \mathcal{F}(G,p)$.

A well-positioned framework $(G,p)$ is \textit{infinitesimally rigid (in $X$)} if every flex is trivial and \textit{infinitesimally flexible (in $X$)} otherwise. We shall define a well-positioned $(G,p)$ framework to be \textit{independent} if the rigidity operator of $G$ at $p$, $df_G(p)$, is surjective (or equivalently, if $G$ is finite, $|E(G)| = \rank df_G(p)$) and define $(G,p)$ to be \textit{dependent} otherwise. If a framework is infinitesimally rigid and independent we shall say that it is \textit{isostatic}. We shall use the convention that any framework with no edges (regardless of placement) is independent and that $(K_1,p)$ is isostatic for any choice of placement $p$. It is immediate that if a framework is independent then its placement is regular, however the reverse does not necessarily hold.

\begin{remark}
In the setting of Euclidean space, all of the above definitions agree with those used in \cite{asiroth} \cite{asirothtwo}.
\end{remark}

\begin{lemma}\label{wecanspan}
Let $(G,p)$ be a finite (possibly not spanning) framework in a $d$-dimensional normed space $ X$ with $|V(G)| \geq d+1$. Suppose $q \in \mathcal{R}(G)$ is full, then the following hold:
\begin{enumerate}[(i)]
\item \label{wecanspan1} If $(G,p)$ is independent then $(G,p)$ is regular and $(G,q)$ is independent.

\item \label{wecanspan2} If $(G,p)$ is infinitesimally rigid then $(G,p)$ is regular, $p$ is full and $(G,q)$ is infinitesimally rigid.
\end{enumerate}
\end{lemma}

\begin{proof}
(\ref{wecanspan1}): As $(G,p)$ is independent then $df_G(p)$ is surjective. As surjective linear maps have maximal possible rank then $(G,p)$ is regular. Since $q$ is regular it follows that $(G,q)$ is independent. 

(\ref{wecanspan2}): As $(G,q)$ is regular then by the Rank-Nullity theorem we have
\begin{align*}
d|V(G)| - \dim \mathcal{T}(p) = \rank df_G(p) \leq \rank df_G(q) \leq d|V(G)| - \dim \mathcal{T}(q),
\end{align*}
thus by Theorem \ref{flextheorem}, $\dim \mathcal{T}(q) \leq \dim \mathcal{T}(p) \leq \dim \Iso(X)$. As $q$ is full then by Theorem \ref{flextheorem}, $\dim \mathcal{T}(q) = \dim \Iso (X)$. It follows that $\dim \mathcal{T}(p) = \dim \Iso (X)$ and thus $p$ is full. From the inequality it also follows that $(G,q)$ is infinitesimally rigid.
\end{proof}

\begin{lemma}\label{ind}
Let $(G,p)$ be a independent framework and $(H,q) \subset (G,p)$, then $(H,q)$ is also independent.
\end{lemma}

\begin{proof}
Choose $a \in \mathbb{R}^{E(H)}$. As $a \times (0)_{vw \in E(G) \setminus E(H)} \in \mathbb{R}^{E(G)}$ and $(G,p)$ is independent there exists $x \in X^{V(G)}$ such that $df_G(p)(x) = a \times (0)_{vw \in E(G) \setminus E(H)}$. We now note $df_H(q)(x|_{V(H)}) =a$ as required.
\end{proof}

The following gives us some necessary and sufficient conditions for infinitesimal rigidity.

\begin{theorem}\cite[Theorem 10]{maxwell}\label{maxwellpaper}
Let finite $(G,p)$ be well-positioned in $X$, then the following hold:
\begin{enumerate}[(i)]
\item If $(G,p)$ is independent then $|E(G)| = (\dim X) |V(G)| - \dim \mathcal{F}(G,p)$.

\item If $(G,p)$ is infinitesimally rigid then $|E(G)| \geq (\dim X) |V(G)| - \dim \mathcal{T}(p)$.
\end{enumerate}
\end{theorem}

The following gives an equivalence for isostaticity.

\begin{proposition}\label{isostatic}
Let $(G,p)$ be a well-positioned framework in $X$. If any two of the following properties hold then so does the third (and $(G,p)$ is isostatic):
\begin{enumerate}[(i)]
\item $|E(G)| = (\dim X)|V(G)| - \dim \mathcal{T}(p)$
\item $(G,p)$ is infinitesimally rigid
\item $(G,p)$ is independent. 
\end{enumerate}
\end{proposition}

\begin{proof}
Apply the Rank-Nullity theorem to the rigidity operator of $G$ at $p$. The result follows the same method as \cite[Lemma 2.6.1.c]{comrig}.
\end{proof}

Using the results from the last section we may now give a stronger result for independent frameworks. 

\begin{corollary}\label{papernecessary}
Let $(G,p)$ be a finite independent framework with $|V(G)|\geq \dim X +1$. Then for all $H \subset G$ with $|V(H)| \geq d+1$ we have $|E(H)| \leq (\dim X)|V(H)| -  \dim \Iso (X)$. If $(G,p)$ is also isostatic then $|E(G)| = (\dim X)|V(G)| -  \dim \Iso (X)$.
\end{corollary}

\begin{proof}
By Lemma \ref{wellpos} and Lemma \ref{regopen2} since $\mathcal{G}(G)$ is an open dense subset of $X^{V(G)}$ and $\mathcal{G}(G)^c$ has measure zero the set $\mathcal{R}(G) \cap \mathcal{G}(G)$ is non-empty, thus we choose $p'$ to be a regular placement of $G$ in general position. Since $(G,p')$ is regular it follows that it is also independent.

Define $q := p'|_{V(H)}$, then $(H,q)$ is in general position. As $(H,q) \subseteq (G,p')$ then by Proposition \ref{ind}, $(H,q)$ is independent; furthermore as $H$ has at least $d+1$ vertices then $q$ is spanning. By Theorem \ref{maxwellpaper} we have $|E(H)| = (\dim X)|V(H)| - \dim \mathcal{F}(H,q)$. By Corollary \ref{spanfull} and Proposition \ref{homeo2.2}, $\dim \mathcal{T}(q) = \dim \Iso(X)$, thus as $\mathcal{T}(q) \subset \mathcal{F}(H,q)$ we have the required inequality.

If $(G,p)$ is also isostatic then by Proposition \ref{isostatic}, $|E(G)| = (\dim X)|V(G)| - \dim \mathcal{T}(p)$. By part \ref{wecanspan2} of Lemma \ref{wecanspan}, $p$ is full, thus by Theorem \ref{flextheorem}, $\dim \mathcal{T}(p) = \dim \Iso (X)$ as required.
\end{proof}

\subsection{Proof of Theorem \ref{asimowroth}}

For a finite graph $G$ we say that a well-positioned framework $(G,p)$ is \textit{constant} if there is a neighbourhood $\mathcal{N}(p) \subset \mathcal{W}(G)$ of $p$ such that $\rank df_G(q) = \rank df_G(p)$ for all $q \in \mathcal{N}(p)$. We shall denote $\mathcal{C}(G)$ to be the subset of $\mathcal{W}(G)$ of constant placements of $G$. 

For Euclidean spaces $\mathcal{R}(G)= \mathcal{C}(G)$ as $\mathcal{R}(G)$ is an open dense subset of $X^{V(G)}$ (see \cite[Section 3]{asiroth} for more details).

\begin{lemma}\label{paperman}
Let $(G,p)$ a constant finite framework in $X$, then there exists an open neighbourhood $U \subset X^{V(G)}$ of $p$ such that $f^{-1}_G [f_G(p)] \cap U$ is a $C^1$-manifold with tangent space $\mathcal{F}(G,p)$ at $p$ and $\mathcal{O}_p \cap U$ is a $C^1$-submanifold of $f^{-1}_G [f_G(p)] \cap U$.
\end{lemma}

\begin{proof}
Since $(G,p)$ is constant $p$ is an interior point of $\mathcal{W}(G)$, so by Proposition \ref{paper2} and Lemma \ref{rigopcont}, $f_G$ is $C^1$-differentiable with constant rank on an open neighbourhood of $p$ in $X^{V(G)}$. By the Constant Rank Theorem (\cite[Theorem 2.5.15]{manifold}) there exists an open neighbourhood $U \subset X^{V(G)}$ of $p$ such that $f^{-1}_G [f_G(p)] \cap U$ is a $C^1$-manifold with tangent space $\ker df_G(p) =\mathcal{F}(G,p)$ at $p$.

By Theorem \ref{homeo2}, $\mathcal{O}_p$ is a smooth manifold. As $\mathcal{O}_p \cap U \subseteq f^{-1}_G[f_G(p)] \cap U \subseteq X^{V(G)}$ and both are $C^1$-submanifolds of $X^{V(G)}$ then the inclusion map $\mathcal{O}_p \cap U \hookrightarrow  f^{-1}_G[f_G(p)] \cap U$ is a $C^1$-embedding, thus $\mathcal{O}_p \cap U$ is a $C^1$-submanifold of $f^{-1}_G [f_G(p)] \cap U$.
\end{proof}

We are now ready to prove Theorem \ref{asimowroth}.

\begin{proof}[Theorem \ref{asimowroth}]
By Lemma \ref{paperman}, $\mathcal{O}_p \cap U$ is a $C^1$-submanifold of $f^{-1}_G [f_G(p)] \cap U$ for some open neighbourhood $U$ of $p$. As manifolds are locally path-connected we may assume we chose $U$ small enough such that $f^{-1}_G [f_G(p)] \cap U$ and $\mathcal{O}_p \cap U$ are path-connected. 

(Infinitesimal rigidity $\Leftrightarrow$ Local rigidity): Since $\mathcal{O}_p \cap U$ is a $C^1$-submanifold of $f^{-1}_G [f_G(p)] \cap U$ we have 
\begin{align*}
&f^{-1}_G [f_G(p)] \cap U' = \mathcal{O}_p \cap U' \text{ for some open neighbourhood $U' \subseteq U$ of $p$} \\
\Leftrightarrow \qquad &T_p (f^{-1}_G[f_G(p)] \cap U )= T_p (\mathcal{O}_p \cap U) \\
\Leftrightarrow \qquad &\mathcal{F}(G,p) = \mathcal{T}(p);
\end{align*}
this is equivalent to saying $(G,p)$ is infinitesimally rigid if and only if $(G,p)$ is locally rigid.

(Continuous rigidity $\Rightarrow$ Local rigidity): Suppose $(G,p)$ is continuously rigid. Choose $q \in f^{-1}_G [f_G(p)] \cap U$, then there exists a continuous path from $p$ to $q$ in $f^{-1}_G[f_G(p)] \cap U$. This implies that we may define a finite flex $\alpha$ of $(G,p)$ such that $\alpha(t_0) =q$ for some $t_0 \in (-\delta,\delta)$. Since $(G,p)$ is continuously rigid then $\alpha$ is trivial and thus a continuous path in $\mathcal{O}_p$. It now follows $q \in \mathcal{O}_p \cap U$ as required.

(Local rigidity $\Rightarrow$ Continuous rigidity): Suppose $(G,p)$ is locally rigid, then there exists $\epsilon >0$ such that $B_\epsilon(p) \subset U$ (the open ball with respect to the normed space $(X^{V(G)}, \| \cdot\|_{V(G)})$) and $f^{-1}_G [f_G(p)] \cap B_\epsilon(p) = \mathcal{O}_p \cap B_\epsilon(p)$. First note that both $f^{-1}_G [f_G(p)]$ and $\mathcal{O}_p$ are invariant under $\Iso (X)$. Choose any $q \in \mathcal{O}_p$, then there exists $g \in \Iso (X)$ such that $g.p=q$. We now note that 
\begin{align*}
\mathcal{O}_p \cap B_\epsilon(q) = g. \left(\mathcal{O}_p \cap B_\epsilon(p)\right) = g. \left( f^{-1}_G [f_G(p)] \cap B_\epsilon(p) \right) = f^{-1}_G [f_G(p)] \cap B_\epsilon(q).
\end{align*}
As this holds for all $q \in \mathcal{O}_p$ then $\mathcal{O}_p$ is open in $f_G^{-1}[f_G(p)]$. By Lemma \ref{homeo2.0}, $\mathcal{O}_p$ is closed in $X^{V(G)}$, thus $\mathcal{O}_p$ is clopen in $f_G^{-1}[f_G(p)]$.

Define $f_G^{-1} [f_G(p)]^\Gamma$ to be the path-connected component of $f_G^{-1} [f_G(p)]$ that contains $p$ with the subspace topology, then the only clopen set in $f_G^{-1} [f_G(p)]^\Gamma$ is itself. Define $\mathcal{O}_p^\Gamma := \mathcal{O}_p \cap f_G^{-1} [f_G(p)]^\Gamma$, then $\mathcal{O}_p^\Gamma$ is clopen since $\mathcal{O}_p$ is clopen. This implies that $\mathcal{O}_p^\Gamma = f_G^{-1} [f_G(p)]^\Gamma$ and so any finite flex $\alpha$ lies in $\mathcal{O}_p$.
\end{proof}

\begin{remark}
Suppose $G$ is any finite graph and $\smooth$ is an open subset of $X$ (an example would be any $\ell^d_q$ space). We note $\mathcal{W}(G)$ will be an open subset of $X^{V(G)}$ and so by Lemma \ref{regopen2}, $\mathcal{R}(G)$ will be an open subset of $X^{V(G)}$. It now follows that every regular placement will be constant, thus by Theorem \ref{asimowroth}, if $(G,p)$ is infinitesimally rigid then it will be continuously and locally rigid also.
\end{remark}

\section{Flexibility of small frameworks and stronger bounds for \texorpdfstring{$\mathcal{T}(p)$}{T(p)}}\label{small section}

For any placement $p$ in $ X$ we shall define $\mathcal{T}_2(p)$ to denote the space of trivial motions of $p$ in $(X,\| \cdot \|_2)$, the unique Euclidean space for $(X,\| \cdot\|)$ as defined in Lemma \ref{paperisolem}. If we refer to just $X$ we shall be referring to the general normed space $(X,\| \cdot\|)$.

\begin{lemma}\label{papertriv}
$\mathcal{T}(p)$ is a linear subspace of $\mathcal{T}_2(p)$.
\end{lemma}

\begin{proof}
By Lemma \ref{paperisolem}, $\Iso (X,\| \cdot \|) \subseteq \Iso (X, \| \cdot \|_2)$. It now follows that $\mathcal{T}(p) \subseteq \mathcal{T}_2(p)$.
\end{proof}

For Euclidean spaces we have the following equality for the dimension of the space of trivial motions for non-spanning placements.

\begin{lemma}\cite[Lemma 2.3.3]{comrig}\label{euclidtrivflex}
Let $(p,S)$ be a placement in $ X$ where $d = \dim X$ and $n$ is the dimension of the affine span of $\{p_v : v \in S \}$. Then 
\begin{align*}
\dim \mathcal{T}_2(p) =\frac{(n+1)(2d - n)}{2}.
\end{align*}
\end{lemma}

We now wish to obtain an upper and lower bound for the dimension of the space of trivial motions for non-spanning placements. To do this we shall first find an upper-bound for when $|S|=2$ in non-Euclidean normed spaces and then use an inductive argument.

\begin{lemma}\label{orblem}
Let $x_0 \in X \setminus \{0\}$ and $\dim X =d$. Then the set
\begin{align*}
\mathcal{O}(x_0) := \{ T(x_0) : T \in \Isolin(X) \}
\end{align*}
is a closed smooth submanifold of $X$; further $\dim \mathcal{O}(x_0) =d-1$ if and only if $X$ is Euclidean.
\end{lemma}

\begin{proof}
Since $\Isolin (X)$ is compact then $\Isolin (X)$ gives rise to a proper Lie group action on $X$ by $x \mapsto T(x)$ for all $T \in \Isolin(X)$, $x \in X$. As $\mathcal{O}(x_0)$ is the orbit of $x_0$ (with respect to $\Isolin (X)$) then by Lemma \ref{orbits}, $\mathcal{O}(x_0)$ is a closed smooth submanifold of $X$. 

First suppose $X$ is Euclidean. By \cite[Corollary 3.3.3]{minkowski} $\Isolin (X)$ acts transitively on $S_{\|x_0\|}[0]$, thus $\mathcal{O}(x_0) = S_{\|x_0\|}[0]$. As the unit sphere of a Euclidean space is the $d$-sphere and $S_{\|x_0\|}[0] = \|x_0\|S_{1}[0]$ we have $\dim \mathcal{O}(x_0) =d-1$.

Now suppose $\dim \mathcal{O}(x_0) =d-1$. If $d=1$ the result is immediate so assume $d >1$. The set $S_{\|x_0\|}[0]$ is a closed connected topological submanifold of $X$ with dimension $d-1$ as it is homeomorphic to the $d$-sphere. Since $\mathcal{O}(x_0) \subset S_{\|x_0\|}[0]$ then $\mathcal{O}(x_0)$ is a closed subset of $S_{\|x_0\|}[0]$. As $\dim \mathcal{O}(x_0) = \dim S_{\|x_0\|}[0]$ it follows from Brouwer's theorem for invariance of domain \cite[Theorem 1.18]{manifoldlee} that the inclusion map $\mathcal{O}(x_0) \hookrightarrow S_{\|x_0\|}[0]$ is an open map and thus $\mathcal{O}(x_0)$ is an open subset of $S_{\|x_0\|}[0]$. As the only clopen non-empty subset of $S_{\|x_0\|}[0]$ is itself we have $\mathcal{O}(x_0) = S_{\|x_0\|}[0]$. This implies $\Isolin (X)$ acts transitively on $S_{\|x_0\|}[0]$, thus by \cite[Corollary 3.3.5]{minkowski} $X$ is Euclidean.
\end{proof}

\begin{lemma}\label{k2flex}
Let $(p,S)$ be a general positioned placement in a $d$-dimensional space $X$ where $|S|=2$. Then $\dim \mathcal{T}(p) \leq 2d-1$ with equality if and only if $X$ is Euclidean.
\end{lemma}

\begin{proof}
By Lemma \ref{papertriv} and Lemma \ref{euclidtrivflex} it follows $\dim \mathcal{T}(p) \leq 2d-1$ and $\dim \mathcal{T}(p) = 2d-1$ if $X$ is Euclidean and $p$ is in general position. 

Now suppose $\dim \mathcal{T}(p) = 2d-1$. Let $S=\{v_1,v_2\}$, then without loss of generality we may assume $p_{v_1}=0$. If $d=1$ the result holds trivially so we may suppose $d>1$. As $p$ is in general position then $p_{v_2} \neq 0$. By Lemma \ref{orblem}, $\mathcal{O}(p_{v_2})$ is a closed smooth submanifold of $X$. It follows from \cite[Lemma 3.3.4]{manifold} that we may identify the tangent space of $\mathcal{O}(p_{v_2})$ at $p_{v_2}$ with a subspace of $X$; by calculation we see that the tangent space of $\mathcal{O}(p_{v_2})$ at $p_{v_2}$ is $\{ u \in X : (0,u) \in \mathcal{T}(p) \}$. Since $\dim \mathcal{T}(p) =d + (d-1)$ and the trivial motions generated by translations form a $d$-dimensional subspace then it follows $\mathcal{O}(p_{v_2})$ has dimension $d-1$. By Lemma \ref{orblem} the space $X$ is Euclidean as required.
\end{proof}

\begin{lemma}\label{quotientlemma}
Let $(p,S)$ be a placement in a $d$-dimensional space $ X$ and $\{p_v :v \in S\}$ have an affine span of dimension $k-1$ where $2 \leq k \leq d$. Suppose $(p',S')$ is a placement where $(p,S) \subset (p',S')$ and the dimension of the affine span of $\{p'_v :v \in S'\}$ is $k$, then
\begin{align*}
\dim \mathcal{T}(p') - \dim \mathcal{T}(p) \leq d-k.
\end{align*}
\end{lemma}

\begin{proof}
Choose $T :=\{v_0, \ldots, v_{k-1}\}$ and $T' := T \cup \{v_{k}\}$ where $v_i \in S$ for $i \in \{0, \ldots ,k-1\}$, $v_{k} \in S'$ and $p'_{v_0}, \ldots, p'_{v_{k}}$ have affine span with dimension $k$. Define the linear restriction map 
\begin{align*}
P: X^{T'} \rightarrow X^{T}, ~ (x_{v_i})_{i=0}^{k} \mapsto (x_{v_i})_{i=0}^{k-1}
\end{align*}
and the placements $(q,T) \subset (p,S)$ and $(q',T') \subset (p',S')$. We note that $P(\mathcal{T}_2(q')) \subseteq \mathcal{T}_2(q)$ and $P(\mathcal{T}(q')) \subseteq \mathcal{T}(q)$. Choose any $u \in \mathcal{T}(q)$, then by Corollary \ref{homeo2.1} there exists $h \in T_\iota \Iso (X)$ such that $u = h.q$. By Corollary \ref{homeo2.1}, $h.q' \in \mathcal{T}(q')$, thus as $P(h.q')= h.q$ we have $P(\mathcal{T}(q')) = \mathcal{T}(q)$. By a similar argument we can see that $P(\mathcal{T}_2(q')) = \mathcal{T}_2(q)$ also.

By the Rank-Nullity theorem applied to $P|_{\mathcal{T}_2(q')}$ and Lemma \ref{euclidtrivflex}
\begin{align*}
\dim \ker P|_{\mathcal{T}_2(q')} = \dim \mathcal{T}_2(q') - \dim \mathcal{T}_2(q) = \frac{(k+1)(2d-k)}{2} - \frac{k(2d-k+1)}{2} = d-k.
\end{align*}
By Lemma \ref{papertriv} and the Rank-Nullity theorem applied to $P|_{\mathcal{T}(q')}$
\begin{align*}
\dim \mathcal{T}(q') - \dim \mathcal{T}(q) = \dim \ker P|_{\mathcal{T}(q')} \leq \dim \ker P|_{\mathcal{T}_2(q')} =d-k.
\end{align*}
By part \ref{diffeo2item2} of Corollary \ref{diffeo2}, $\mathcal{T}(q) \cong \mathcal{T}(p)$ and $\mathcal{T}(q') \cong \mathcal{T}(p')$ and so the result follows.
\end{proof}

\begin{proposition}\label{flextheorem2}
Let $(p,S)$ be a placement in a $d$-dimensional normed space $X$ where $\{p_v :v \in S \}$ has an affine span of dimension $1 \leq n \leq d$. Then
\begin{align*}
\dim \mathcal{T}(p) \leq \frac{(n+1)(2d - n)}{2} 
\end{align*}
with equality if and only if $X$ is Euclidean.
\end{proposition}

\begin{proof}
If $X$ is Euclidean then the result follows by Lemma \ref{euclidtrivflex}.

Suppose $X$ is non-Euclidean. If $n=1$ then the result follows by Lemma \ref{k2flex} and part \ref{diffeo2item2} of Corollary \ref{diffeo2}. Let $n >1$ and suppose the theorem holds for all $m= 1, \ldots, n-1$. Choose $v_1, \ldots, v_n \in S$ so that $p_{v_1}, \ldots, p_{v_n}$ are affinely independent and define $T:= \{v_1, \ldots, v_n \}$. If we define $q := p|_T$ then $\{q_v : v = v_1, \ldots, v_n\}$ has an affine span with dimension $n-1$, thus by assumption $\dim \mathcal{T}(q) < \frac{n(2d - n+1)}{2}$. By Lemma \ref{quotientlemma} it follows that 
\begin{align*}
\dim \mathcal{T}(p) \leq \dim \mathcal{T}(q) +d -n < \frac{n(2d - n+1)}{2} +d-n = \frac{(n+1)(2d - n)}{2}.
\end{align*}
\end{proof}

We define a framework $(G,p)$ in a $d$-dimensional normed space to be \textit{small} if $|V(G)| \leq d +1$. The following is a well known result for Euclidean spaces.

\begin{proposition}\label{paperbarsimplex}
Let $X$ be $d$-dimensional Euclidean space and $(G,p)$ be a small well-positioned framework, then $(G,p)$ is isostatic if and only if $G$ is a complete graph and $p$ is in general position.
\end{proposition}

\begin{proof}
($\Rightarrow$): If $(G,p)$ is isostatic it follows by Proposition \ref{isostatic} and Lemma \ref{euclidtrivflex} that $|E(G)| = \frac{|V(G)|(|V(G)|-1)}{2}$, thus $G$ is a complete graph. 

Suppose $p$ is not in general position, then there exists distinct vertices $v_0, \ldots, v_n, w \in V(G)$ such that the affine span of $p_{v_0}, \ldots, p_{v_n}$ is $Y \subset X$, the affine span of $\{p_v :v \in V(G)\}$. Define for each $v \in V(G)$ the vector $u_v$ where $u_v=0$ if $v \neq w$ and $u_w \in (Y - p_{v_0})^\perp \setminus \{0\}$, then $u := (u_v)_{v \in V(G)} \in X^{V(G)}$ is a infinitesimal flex of $(G,p)$. If $u$ is trivial then by Corollary \ref{homeo2.1} there exists an affine map $g \in T_{\iota} \Iso (X)$ such that $g.p = u$. As $p_{v_0}, \ldots, p_{v_n}$ is an affine basis of $Y$ and $g(p_{v_i})=0$ for all $i=0,1, \ldots,n$ then $g(y)=0$ for all $y \in Y$. However $g(p_w) \neq 0$ and so no such affine map $g$ may exist. This implies $u$ is a non-trivial flex which contradicts the infinitesimal rigidity of $(G,p)$.

($\Leftarrow$): This follows from \cite[Theorem 2.4.1.d]{comrig}.
\end{proof}

Using Theorem \ref{flextheorem} we can now state our own result for small frameworks for non-Euclidean spaces.

\begin{theorem}\label{smallflexible}
Let $(G,p)$ be a small well-positioned framework with in a $d$-dimensional non-Euclidean normed space $X$, then $(G,p)$ is infinitesimally flexible or $|V(G)|=1$.
\end{theorem}

\begin{proof}
If $|V(G)|=1$ then $(G,p)$ is infinitesimally rigid by definition. Suppose $|V(G)| \geq 2$ and the affine span of $\{p_v:v \in V(G)\}$ has dimension $n$. Define the map $f: \mathbb{R} \rightarrow \mathbb{R}$ where 
\begin{align*}
f(x) = \frac{(x+1)(2d- x)}{2}.
\end{align*}
We note that $f$ is increasing on the interval $[0, d-1]$ and $f(d-1)=f(d)$, thus it follows that $f(|V(G)|-1) \geq f(n)$. We note
\begin{align*}
|E(G)| &\leq \frac{|V(G)|(|V(G)|-1)}{2} \\
&= d|V(G)| - f(|V(G)|-1) \\
&\leq d|V(G)| - f(n) \\
&< d|V(G)| - \dim \mathcal{T}(p) \text{ by Lemma \ref{flextheorem2}}
\end{align*}
thus by Theorem \ref{maxwellpaper}, $(G,p)$ is infinitesimally flexible. 
\end{proof}

\begin{remark}[Final remark]
We remark that combinatorial characterisations of infinitesimal rigidity have recently been obtained in some normed space contexts. However other forms of rigidity, such as redundant and global rigidity, have not yet been explored in general normed spaces.
\end{remark}

\end{document}